\documentclass[11pt,reqno]{amsart}

\usepackage[a4paper,margin=1.08in]{geometry}
\usepackage[T1]{fontenc}
\usepackage[utf8]{inputenc}
\usepackage{lmodern}
\usepackage{microtype}
\usepackage{amsmath,amssymb,amsthm,mathtools}
\usepackage{enumitem}
\usepackage{aliascnt}
\usepackage[hidelinks]{hyperref}
\usepackage[nameinlink,capitalise,noabbrev]{cleveref}
\usepackage{parskip}
\usepackage{tikz}

\newaliascnt{proposition}{theorem}
\newtheorem{proposition}[proposition]{Proposition}
\aliascntresetthe{proposition}
\crefname{proposition}{proposition}{propositions}
\Crefname{proposition}{Proposition}{Propositions}

\newaliascnt{lemma}{theorem}
\newtheorem{lemma}[lemma]{Lemma}
\aliascntresetthe{lemma}
\crefname{lemma}{lemma}{lemmas}
\Crefname{lemma}{Lemma}{Lemmas}

\newaliascnt{corollary}{theorem}
\newtheorem{corollary}[corollary]{Corollary}
\aliascntresetthe{corollary}
\crefname{corollary}{corollary}{corollaries}
\Crefname{corollary}{Corollary}{Corollaries}

\newtheorem{maintheorem}{Theorem}
\crefname{maintheorem}{theorem}{theorems}
\Crefname{maintheorem}{Theorem}{Theorems}

\theoremstyle{definition}
\newaliascnt{definition}{theorem}
\newtheorem{definition}[definition]{Definition}
\aliascntresetthe{definition}
\crefname{definition}{definition}{definitions}
\Crefname{definition}{Definition}{Definitions}

\theoremstyle{remark}
\newaliascnt{remark}{theorem}
\newtheorem{remark}[remark]{Remark}
\aliascntresetthe{remark}
\crefname{remark}{remark}{remarks}
\Crefname{remark}{Remark}{Remarks}

\newcommand{\PP}{\mathbb P}
\newcommand{\OO}{\mathcal O}
\newcommand{\EE}{\mathcal E}

\newcommand{\II}{\mathcal I}
\newcommand{\Sym}{\operatorname{Sym}}
\newcommand{\Bs}{\operatorname{Bs}}

\newcommand{\Bl}{\operatorname{Bl}}

\newcommand{\wt}{\widetilde}

\title[Fujita freeness for projectivized toric vector bundles]
{Fujita freeness for projectivized toric vector bundles}

\author{Antonio Laface}
\date{August 2026}


\begin{document}

\begin{abstract}
Let $X$ be a smooth projective toric variety of dimension $n\geq1$ over an algebraically closed field of characteristic zero, let $\EE$ be a toric vector bundle of rank $r\geq2$, and let $\pi\colon Y=\PP_X(\EE)\to X$ be the projective bundle of one-dimensional quotients.  Write an ample line bundle on $Y$ as $A=\OO_Y(a)\otimes\pi^*L$, with $a\geq1$.
We record a blow-up argument proving that $K_Y+mA$ is globally generated whenever an integer $m$ satisfies $ma\geq r$ and $m\delta(A)>n$, where $\delta(A)$ is a positive integer obtained from the degrees of $A$ on the invariant quotient sections over the torus-invariant curves of $X$.  In particular, $K_Y+mA$ is globally generated for $m\geq n+1$ and $ma\geq r$.
Consequently every projectivized toric vector bundle satisfies Fujita's freeness conjecture.  The uniform bound is sharp.  We also formulate the result as a global-generation theorem for adjoint symmetric powers of $\EE$ and explain its relation with the Seshadri-constant results of Hering--Musta\c{t}\u{a}--Payne and Fulger--Murayama.
ChatGPT (OpenAI) was used to assist with mathematical discussion, language, and bibliographic searches.

\end{abstract}

\maketitle

\section*{Introduction}

Let $Z$ be a smooth projective variety of dimension $d$ and let $H$ be an ample line bundle on $Z$.  Fujita's freeness conjecture predicts that $K_Z+mH$ is globally generated for every $m\geq d+1$.  The conjecture is known for toric varieties \cite{Mustata}.  Altmann--Ilten proved it for Gorenstein complexity-one $T$-varieties with rational singularities \cite{AltmannIlten}, which in particular covers projectivizations of rank-two toric vector bundles.  More recently, George--Manon established Fujita-type results for several classes of projectivized toric vector bundles \cite{GeorgeManon}.
The positivity results of Hering--Musta\c{t}\u{a}--Payne \cite{HMP}, combined with the adjoint jet-separation theorem of Fulger--Murayama \cite{FulgerMurayama}, also give the case of polarizations of the form $\OO_{\PP(\EE)}(1)\otimes\pi^*L$. For a general polarization $\OO_{\PP(\EE)}(a)\otimes\pi^*L$, the natural object on the base is the formal $\mathbb Q$-twist $\EE\langle L/a\rangle$ \cite[Section~6.2]{Lazarsfeld}.  The argument below works directly on the projective bundle and treats arbitrary $a$ and arbitrary rank.

Let $X$ be a smooth projective toric variety of dimension $n$, let $\EE$ be a toric vector bundle of rank $r\geq2$, and set $Y=\PP_X(\EE)$.  We use Grothendieck's quotient convention.  For an ample line bundle
\[
A=\OO_Y(a)\otimes\pi^*L,
\]
with $a\geq1$, we define a positive integer $\delta(A)$ from the degrees of $A$ on the invariant quotient sections over the torus-invariant curves of $X$.
Our first main result is the following refinement of the Fujita bound.

\begin{maintheorem}
\label{thm:main}
Let $X$ be a smooth projective toric variety of dimension $n\geq1$, let $\EE$ be a toric vector bundle of rank $r\geq2$, and let $Y=\PP_X(\EE)$.  Suppose that $A=\OO_Y(a)\otimes\pi^*L$ is ample.  If an integer $m$ satisfies $ma\geq r$ and $m\delta(A)>n$, then 
$K_Y+mA$
is globally generated.
\end{maintheorem}

Since $\delta(A)$ is a positive integer, \Cref{thm:main} immediately implies that $K_Y+mA$ is globally generated whenever $m\geq n+1$ and $ma\geq r$.  Equivalently, it is enough to take $m\geq\max\{n+1,\lceil r/a\rceil\}$.  In particular, since $\dim Y=n+r-1$, every projectivized toric vector bundle over a smooth projective toric variety satisfies Fujita's freeness conjecture; see \Cref{cor:uniform,cor:fujita}.  The uniform bound is sharp.
The proof is based on a blow-up of the fiber over a torus-fixed point.  The positivity criterion of Hering--Musta\c{t}\u{a}--Payne controls the subtraction of the exceptional divisor, while Kodaira vanishing gives surjectivity of restriction to the fixed fiber.  Global generation on that fiber is then elementary, and torus invariance of the base locus gives global generation on $Y$.
The same argument also yields a statement for adjoint symmetric powers on the base.

\begin{maintheorem}
\label{thm:vector-bundle}
Under the hypotheses of \Cref{thm:main}, let $m$ be an integer satisfying $ma\geq r$ and $m\delta(A)>n$.  Then
\[
\omega_X\otimes\det\EE\otimes\Sym^{ma-r}\EE\otimes L^{\otimes m}
\]
is globally generated.
\end{maintheorem}
In particular, if $\EE$ is an ample toric vector bundle of rank $r$ on a smooth projective toric $n$-fold, then $\omega_X\otimes\det\EE\otimes\Sym^q\EE$ is globally generated whenever $(q+r)\delta(\EE)>n$, and hence whenever $q+r\geq n+1$; see \Cref{cor:ample-E}.

\subsection*{Acknowledgements}
This note grew out of a course given by Nathan Ilten at the Banff
International Research Station--Casa Matem\'atica Oaxaca workshop
\href{https://www.birs.ca/events/2026/5-day-workshops/26w5505}
{\emph{Symmetries, Invariants and Singularities in Algebraic Geometry
with a View towards Latin America}}.  The author thanks Nathan Ilten for
helpful discussions on the questions considered in this note.  The author
was partially supported by Proyecto FONDECYT Regular No.~1230287.

\section{Proof of Theorem 1}

\subsection{Invariant edge degrees}

Throughout the paper, the ground field is algebraically closed of characteristic zero.  We assume $n\geq1$; when $n=0$, the statement is the elementary computation on $Y\simeq\PP^{r-1}$.

Fix an ample line bundle $A=\OO_Y(a)\otimes\pi^*L$, with $a\geq1$.  Let $C\subset X$ be a torus-invariant irreducible curve.  Since $C\simeq\PP^1$, the restriction of $\EE$ splits as $\EE|_C\simeq\bigoplus_{j=1}^r\OO_{\PP^1}(b_{C,j})$.  Each quotient $\EE|_C\twoheadrightarrow\OO_{\PP^1}(b_{C,j})$ defines a section $s_{C,j}\colon C\to\PP_C(\EE|_C)\subseteq Y$.  Along this section the tautological quotient line bundle has degree $b_{C,j}$, and therefore
\begin{equation}
\label{eq:edge-degree}
 A\cdot s_{C,j}(C)=a b_{C,j}+L\cdot C.
\end{equation}
The left-hand side is a positive integer because $A$ is ample.

\begin{definition}
Let $x\in X$ be a torus-fixed point.  Define
\[
 \delta_x(A):=\min\bigl\{a b_{C,j}+L\cdot C\;\big|\; C\text{ invariant},\ x\in C,\ 1\leq j\leq r\bigr\}.
\]
We also put $\delta(A):=\min_{x\in X^T}\delta_x(A)$.  Thus $\delta_x(A)$ and $\delta(A)$ are positive integers.
\end{definition}
By the proof of \cite[Proposition~3.2]{HMP}, together with the $\mathbb Q$-twisted positivity criterion of \cite[Remark~3.1]{HMP}, the quantity $\delta_x(A)/a$ is the local Seshadri constant at $x$ of the ample $\mathbb Q$-twisted toric vector bundle $\EE\langle L/a\rangle$.  We will only need the elementary restriction calculation underlying that statement.

\subsection{Ampleness after blowing up a fixed fiber}

Fix a torus-fixed point $x\in X$, and let $p\colon\wt X=\Bl_xX\to X$ be the blow-up, with exceptional divisor $F$.  Since $\pi$ is flat, blow-up commutes with this base change.  Hence
\[
 \wt Y:=Y\times_X\wt X\simeq\PP_{\wt X}(p^*\EE)\simeq\Bl_{Y_x}Y.
\]
We denote the two projections by $\rho\colon\wt Y\to Y$ and $\wt\pi\colon\wt Y\to\wt X$, and write $G=\wt\pi^*F$ for the exceptional divisor of $\rho$.

When $n=1$, the blow-up of the Cartier divisor $x\subset X$ is the identity; in this case one reads the notation above as $\wt X=X$, $F=x$, $\wt Y=Y$, $\rho=\mathrm{id}_Y$, and $G=Y_x$.

\begin{lemma}[Fixed-fiber Seshadri estimate]
\label{lem:blowup-ample}
For every rational number $\lambda$ satisfying $0<\lambda<\delta_x(A)$, the $\mathbb Q$-divisor $\rho^*A-\lambda G$ is ample on $\wt Y$.
\end{lemma}

\begin{proof}
Consider on $\wt X$ the $\mathbb Q$-twisted toric vector bundle
\[
 \mathcal V_\lambda:=p^*\EE\left\langle\frac{p^*L-\lambda F}{a}\right\rangle.
\]
Its tautological $\mathbb Q$-divisor on $\PP_{\wt X}(p^*\EE)=\wt Y$ is
\[
 \OO_{\wt Y}(1)+\wt\pi^*\left(\frac{p^*L-\lambda F}{a}\right)=\frac1a(\rho^*A-\lambda G).
\]
Thus it is enough to prove that $\mathcal V_\lambda$ is ample.  By \cite[Theorem~2.1 and Remark~3.1]{HMP}, ampleness of a $\mathbb Q$-twisted toric vector bundle can be checked on the invariant curves of the base.

There are three kinds of torus-invariant curves on $\wt X$, and the verification is immediate in each case.  If the curve is the strict transform of an invariant curve $C$ not passing through $x$, then the degrees of the summands of $\mathcal V_\lambda$ are unchanged and equal to $(a b_{C,j}+L\cdot C)/a$, which is positive by \eqref{eq:edge-degree}.

If instead the curve is the strict transform $\wt C$ of an invariant curve $C$ passing through $x$, then $F\cdot\wt C=1$.  The corresponding summand degree is $(a b_{C,j}+L\cdot C-\lambda)/a$.  By the definition of $\delta_x(A)$ this is at least $(\delta_x(A)-\lambda)/a$, and is therefore positive.

Finally, let $\ell\subset F$ be an invariant curve contained in the exceptional divisor.  The restriction $p^*\EE|_\ell$ is trivial, while $p^*L\cdot\ell=0$ and $F\cdot\ell=-1$.  Every summand of $\mathcal V_\lambda|_\ell$ therefore has degree $\lambda/a>0$.

Thus $\mathcal V_\lambda$ restricts to an ample $\mathbb Q$-twisted vector bundle on every invariant curve of $\wt X$.  The criterion of Hering--Musta\c{t}\u{a}--Payne proves that $\mathcal V_\lambda$ is ample, and therefore so is $\rho^*A-\lambda G$.
\end{proof}

\subsection{Restriction to a fixed fiber}

For an integer $m$, set $D_m:=K_Y+mA$.

\begin{proposition}
\label{prop:restriction}
Let $x\in X$ be a torus-fixed point.  If $m\delta_x(A)>n$, then
\[
 H^1\bigl(Y,D_m\otimes\II_{Y_x}\bigr)=0.
\]
Consequently, the restriction map $H^0(Y,D_m)\to H^0(Y_x,D_m|_{Y_x})$ is surjective.
\end{proposition}

\begin{proof}
The center $Y_x\subset Y$ has codimension $n$, so the canonical divisor of the blow-up is
\[
 K_{\wt Y}=\rho^*K_Y+(n-1)G.
\]
The assumption $m\delta_x(A)>n$ says precisely that $n/m$ lies strictly between $0$ and $\delta_x(A)$.  Applying \Cref{lem:blowup-ample} with $\lambda=n/m$ shows that the integral divisor $m\rho^*A-nG$ is ample.

This is exactly the positivity needed for Kodaira vanishing, because
\begin{align*}
 \rho^*D_m-G
 &=\rho^*(K_Y+mA)-G\\
 &=K_{\wt Y}+\bigl(m\rho^*A-nG\bigr).
\end{align*}
Hence
\[
 H^1\bigl(\wt Y,\rho^*D_m-G\bigr)=0.
\]

For the blow-up of a smooth center one has $\rho_*\OO_{\wt Y}(-G)=\II_{Y_x}$ and $R^i\rho_*\OO_{\wt Y}(-G)=0$ for $i>0$.  By the projection formula and Leray, the vanishing above is equivalent to
\[
 H^1\bigl(Y,D_m\otimes\II_{Y_x}\bigr)=0.
\]
The desired surjectivity now follows from the standard exact sequence
\[
 0\longrightarrow D_m\otimes\II_{Y_x}\longrightarrow D_m\longrightarrow D_m|_{Y_x}\longrightarrow0.
\]
\end{proof}

\subsection{The freeness theorem}

For the quotient convention, the canonical bundle formula is
\begin{equation}
\label{eq:canonical}
 K_Y=-r\,\OO_Y(1)+\pi^*(K_X+\det\EE).
\end{equation}
Equivalently,
\begin{equation}
\label{eq:Dm}
 D_m=(ma-r)\OO_Y(1)+\pi^*(K_X+\det\EE+mL).
\end{equation}

\begin{proof}[Proof of \Cref{thm:main}]
Let $x\in X$ be a torus-fixed point.  Since $\delta(A)\leq\delta_x(A)$, \Cref{prop:restriction} gives a surjection
\[
 H^0(Y,D_m)\longrightarrow H^0(Y_x,D_m|_{Y_x}).
\]
By \eqref{eq:Dm}, the restriction to the fiber is $D_m|_{Y_x}\simeq\OO_{\PP^{r-1}}(ma-r)$.  The inequality $ma\geq r$ therefore implies that $D_m|_{Y_x}$ is globally generated.  Since every section on the fiber lifts to a global section of $D_m$, the line bundle $D_m$ has no base point anywhere on the fiber $Y_x$.  The same conclusion holds for the fiber over every torus-fixed point of $X$.

It remains to pass from these special fibers to all of $Y$.  The line bundle $D_m$ admits a torus linearization, so its base locus $\Bs(D_m)$ is a closed torus-invariant subset of the projective variety $Y$.  If the base locus were nonempty, the closure of a torus orbit contained in it would contain a torus-fixed point.  Every fixed point of $Y$ lies over a fixed point of $X$, contradicting the fact that $D_m$ is generated along every fixed fiber.  Hence $\Bs(D_m)=\varnothing$.
\end{proof}

\begin{corollary}[Uniform bound]
\label{cor:uniform}
Under the hypotheses of \Cref{thm:main}, the line bundle $K_Y+mA$ is globally generated whenever an integer $m$ satisfies $m\geq n+1$ and $ma\geq r$.  Equivalently, it is globally generated for every
\[
 m\geq\max\left\{n+1,\left\lceil\frac{r}{a}\right\rceil\right\}.
\]
\end{corollary}

\begin{proof}
Since $\delta(A)$ is a positive integer, $\delta(A)\geq1$.  Thus $m\geq n+1$ implies $m\delta(A)>n$, and \Cref{thm:main} applies.
\end{proof}

\begin{corollary}[Fujita freeness]
\label{cor:fujita}
Every projectivized toric vector bundle over a smooth projective toric variety satisfies Fujita's freeness conjecture.
\end{corollary}

\begin{proof}
If $n=0$, then $Y\simeq\PP^{r-1}$ and the assertion is elementary.  If $r=1$, then $Y\simeq X$ and the assertion is the toric case of Fujita's freeness conjecture \cite{Mustata}.  We may therefore assume $n\geq1$ and $r\geq2$.  We have $\dim Y=n+r-1$.  If $m\geq\dim Y+1=n+r$, then certainly $m\geq n+1$ and, since $a\geq1$, one also has $ma\geq m\geq r$.  The result follows from \Cref{cor:uniform}.
\end{proof}

\begin{remark}[Status and relation with previous work]
To the best of our knowledge, Fujita freeness has not previously been
established for arbitrary projectivized toric vector bundles.  Several
important special cases are known.  In particular, the rank-two case
follows from the work of Altmann--Ilten on complexity-one $T$-varieties
\cite{AltmannIlten}, while George--Manon prove Fujita-type results for
several classes of projectivized toric vector bundles \cite{GeorgeManon}.

Over $\mathbb C$, the case $a=1$ of \Cref{thm:main} is already a direct
consequence of the computation of Seshadri constants for toric vector
bundles by Hering--Musta\c{t}\u{a}--Payne and the adjoint jet-separation
theorem of Fulger--Murayama \cite[Proposition~3.2]{HMP} and \cite[Proposition~6.7]{FulgerMurayama};
the corresponding characteristic-zero statement follows by spreading
out and faithfully flat base change.  See also
\Cref{sec:vector-bundle-form}.  For arbitrary $a$, the natural object is
the formal twist $\EE\langle L/a\rangle$.  The blow-up argument above
may be viewed as an integral realization, in the toric setting, of the
same Seshadri-constant mechanism.  Its point is that it treats an
arbitrary polarization $\OO_Y(a)\otimes\pi^*L$ directly and yields
Fujita freeness for projectivizations of toric vector bundles of
arbitrary rank.
\end{remark}

\section{Proof of Theorem 2}

\subsection{The vector-bundle formulation}
\label{sec:vector-bundle-form}

Put $q=ma-r$.  When $q\geq0$, \eqref{eq:Dm} and the projective bundle formula give
\[
 \pi_*\OO_Y(D_m)\simeq\omega_X\otimes\det\EE\otimes\Sym^q\EE\otimes L^{\otimes m}.
\]
The proof of \Cref{thm:main} yields more than base-point-freeness of $D_m$.

\begin{proof}[Proof of \Cref{thm:vector-bundle}]
Since $q\geq0$, the projective bundle formula gives
$R^i\pi_*\OO_Y(q)=0$ for $i>0$.  Cohomology and base change therefore
yields a natural identification
\[
 (\pi_*\OO_Y(q))\otimes k(x)
 \simeq H^0\bigl(\PP(\EE_x),\OO_{\PP(\EE_x)}(q)\bigr).
\]
After tensoring with $(\omega_X\otimes\det\EE\otimes L^{\otimes m})_x$,
the target of the restriction map in \Cref{prop:restriction} identifies
with the fiber at $x$ of
$\omega_X\otimes\det\EE\otimes\Sym^q\EE\otimes L^{\otimes m}$.
Under this identification, the restriction map is precisely the
evaluation map of that vector bundle at $x$.  It is therefore surjective
at every torus-fixed point.  The non-generation locus is closed and
torus-invariant, so it is empty by the same fixed-point argument used in
the proof of \Cref{thm:main}.
\end{proof}

Taking $A=\OO_{\PP(\EE)}(1)$ gives a particularly clean statement.

\begin{corollary}
\label{cor:ample-E}
Let $\EE$ be an ample toric vector bundle of rank $r$ on a smooth projective toric $n$-fold $X$.  Let
\[
 \delta(\EE):=\min_{C,j} b_{C,j},
\]
where the minimum is taken over all invariant curves $C$ and all summands in the splitting of $\EE|_C$.  If $q\geq0$ and $(q+r)\delta(\EE)>n$, then $\omega_X\otimes\det\EE\otimes\Sym^q\EE$ is globally generated.  In particular, it is globally generated as soon as $q+r\geq n+1$.
\end{corollary}

\begin{proof}
Apply \Cref{thm:vector-bundle} with $a=1$, $L=\OO_X$, and $m=q+r$.
\end{proof}

\begin{remark}[Comparison with Fulger--Murayama]
At a torus-fixed point $x$, Hering--Musta\c{t}\u{a}--Payne prove that $\varepsilon(\EE;x)=\delta_x(\EE)$ for a nef toric vector bundle \cite[Proposition~3.2]{HMP}.  Over $\mathbb C$, Fulger--Murayama prove that, for an ample vector bundle $V$ of rank $r$ on a smooth projective $n$-fold, the bundle $\omega_X\otimes\det V\otimes\Sym^qV$ is generated at $x$ whenever $\varepsilon(V;x)>n/(q+r)$ \cite[Proposition~6.7]{FulgerMurayama}.  Applying this at the torus-fixed points and then using torus invariance gives \Cref{cor:ample-E} directly over $\mathbb C$; the characteristic-zero statement follows by spreading out and faithfully flat base change.  More generally, when $a=1$ and $A=\OO_Y(1)\otimes\pi^*L$, one applies the same argument to the ample toric vector bundle $\EE\otimes L$.

For arbitrary $a$, the natural object is the formal twist $\EE\langle L/a\rangle$.  The proof of \Cref{thm:main} may be viewed as an integral blow-up realization of the same Seshadri-constant argument, without requiring a separate adjoint theorem for formal $\mathbb Q$-twists.
\end{remark}

\section{Sharpness}

The two numerical conditions in \Cref{cor:uniform} cannot be improved uniformly.

\begin{proposition}
\label{prop:sharp}
For every $n\geq1$, $r\geq2$, and $a\geq1$, there are $X$, $\EE$, and $A$ as above such that $K_Y+mA$ is globally generated if and only if $m\geq n+1$ and $ma\geq r$.
\end{proposition}

\begin{proof}
Take $X=\PP^n$ and $\EE=\OO_{\PP^n}^{\oplus r}$.  Then $Y\simeq\PP^n\times\PP^{r-1}$.  Let
\[
 A=\operatorname{pr}_1^*\OO_{\PP^n}(1)\otimes\operatorname{pr}_2^*\OO_{\PP^{r-1}}(a),
\]
which is equivalently $A=\OO_Y(a)\otimes\pi^*\OO_{\PP^n}(1)$.  In this case
\[
 K_Y+mA\simeq\OO_{\PP^n\times\PP^{r-1}}(m-n-1,ma-r).
\]
A line bundle $\OO(u,v)$ on the product is globally generated if and only if $u,v\geq0$.  Hence the two displayed numerical conditions are both necessary and sufficient.
\end{proof}

The bound in \Cref{cor:ample-E} is likewise sharp.  Indeed, for $X=\PP^n$ and $\EE=\OO_{\PP^n}(1)^{\oplus r}$ one has
\[
 \omega_X\otimes\det\EE\otimes\Sym^q\EE\simeq\OO_{\PP^n}(q+r-n-1)^{\oplus\binom{q+r-1}{r-1}}.
\]
Thus global generation holds exactly when $q+r\geq n+1$.

\begin{remark}
The determinant factor in \Cref{thm:vector-bundle} is essential. An ample toric vector bundle need not be globally generated; explicit examples were constructed by Di~Rocco--Jabbusch--Smith \cite{DRJS}.  More strongly, N\o dland constructs, for every positive integer $k$, an ample rank-two toric vector bundle $\EE_k$ on a smooth complete toric surface $X$ such that $\Sym^k\EE_k$ is not globally generated \cite[Section~9]{Nodland}.  Since $\Sym^k\EE_k$ is a quotient of $\EE_k^{\otimes k}$, the tensor power $\EE_k^{\otimes k}$ is not globally generated either.

The toric surface $X$ used in N{\o}dland's construction is obtained from
$\PP^1\times\PP^1$ by two successive toric blow-ups, the second centered
at a torus-fixed point on the exceptional divisor of the first.  Its
invariant curves have self-intersection at least $-2$.  Hence $-K_X$ is
nef and, since $X$ is toric, globally generated.
It follows that neither $\omega_X\otimes\Sym^k\EE_k$ nor $\omega_X\otimes\EE_k^{\otimes k}$ is globally generated.  Indeed, global generation of either bundle, after tensoring with the globally generated line bundle $\omega_X^{-1}$, would imply global generation of $\Sym^k\EE_k$ or $\EE_k^{\otimes k}$, respectively.  Thus the natural tensor-power analogue of the adjoint statement fails even for ample rank-two toric vector bundles on smooth toric surfaces; it would imply the corresponding symmetric-power statement because $\Sym^k\EE_k$ is a quotient of $\EE_k^{\otimes k}$.
\end{remark}

\end{document}